\documentclass[11pt]{article}
\usepackage[T1]{fontenc}
\usepackage[utf8]{inputenc}
\usepackage{lmodern}
\usepackage[margin=1in]{geometry}
\usepackage{amsmath,amssymb,amsthm}
\usepackage{microtype}
\usepackage[hidelinks]{hyperref}

\newtheorem{theorem}{Theorem}[section]
\newtheorem{lemma}[theorem]{Lemma}

\newtheorem{corollary}[theorem]{Corollary}
\theoremstyle{definition}
\newtheorem{definition}[theorem]{Definition}
\theoremstyle{remark}
\newtheorem{remark}[theorem]{Remark}
\numberwithin{equation}{section}
\allowdisplaybreaks[1]
\hypersetup{
  pdftitle={A greedy approximation proof of Bourgain's Lambda(p) selection theorem},
  pdfauthor={Will Burstein, Alex Iosevich and Ben Krause}
}

\title{Bourgain’s $\Lambda(p)$ selection theorem: a greedy proof with polynomial failure bounds}
\author{Will Burstein \and Alex Iosevich \and Ben Krause}
\date{}

\begin{document}
\maketitle

\begin{abstract}
We give a self-contained proof of Bourgain's finite $\Lambda(p)$ selection
theorem with an explicit probability bound. For every $p>2$ and $N\ge2$,
given $N$ pairwise orthogonal functions bounded by one on a probability
space, a uniformly chosen subsystem of cardinality
$\lceil N^{\frac{2}{p}}\rceil$ satisfies the $\Lambda(p)$ inequality with
probability at least $1-\frac{1}{N}$, with a constant depending only on
$p$. The inequality holds simultaneously for all complex coefficient
vectors. More generally, for every fixed $A>0$, the success probability
is at least $1-\frac{1}{N^A}$ with a constant depending only on $p$ and
$A$, and independent of $N$. The proof uses a greedy approximation in
$L^r$, with $r>2$, whose potential decreases throughout all approximation
scales. A single weighted bound on the cumulative number of updates
controls a union bound over discrete update histories, and increasing the
weight assigned to each history gives the prescribed failure exponent
without changing the cardinality exponent. 
\end{abstract}

\section{Introduction}

Let $(X,\mu)$ be a probability space, so that $\mu(X)=1$. For
$1\le q<\infty$, write
\[
  \|h\|_q=\left(\int_X|h(x)|^q\,d\mu(x)\right)^{\frac{1}{q}},
\]
and let $\|h\|_\infty$ denote the essential supremum of $|h|$.
For a positive integer $N$, set $[N]=\{1,\ldots,N\}$.
We call functions $\phi_1,\ldots,\phi_N$ pairwise orthogonal if
$\int_X\phi_i(x)\overline{\phi_j(x)}\,d\mu(x)=0$
whenever $i\ne j$.
They need not have $L^2$ norm equal to one. We allow complex-valued
functions and complex coefficients throughout.
We write $\mathbb{P}$ and $\mathbb{E}$ for probability and expectation
with respect to the random selection of indices.

Orthogonality and the bounds $\|\phi_i\|_\infty\le1$ immediately give
\[
  \left\|\sum_{i=1}^N a_i\phi_i\right\|_2^2
  =\sum_{i=1}^N|a_i|^2\|\phi_i\|_2^2
  \le\sum_{i=1}^N|a_i|^2.
\]
For $p>2$, the corresponding $L^p$ estimate can fail with a constant
independent of $N$. Bourgain's selection theorem shows that one can
always retain $N^{\frac{2}{p}}$ functions, up to rounding, and obtain such an
estimate on the resulting subsystem \cite{Bourgain}.

\begin{theorem}[Bourgain's theorem with a probability bound]\label{thm:main}
Let $p>2$, let $N$ be a positive integer, and let
$\phi_1,\ldots,\phi_N$ be pairwise orthogonal functions on a probability
space $(X,\mu)$ satisfying $\|\phi_i\|_\infty\le1$ for every $i\in[N]$.
There is a constant $C_p$, depending only on $p$, and a set
$S\subset[N]$ with
\[
  |S|=\left\lceil N^{\frac{2}{p}}\right\rceil
\]
such that
\begin{equation}\label{eq:main}
  \left\|\sum_{i\in S}a_i\phi_i\right\|_p
  \le C_p\left(\sum_{i\in S}|a_i|^2\right)^{\frac{1}{2}}
\end{equation}
for every choice of complex coefficients $(a_i)_{i\in S}$.
Moreover, if $N\ge2$ and $S$ is chosen uniformly from all subsets of
$[N]$ of cardinality $\lceil N^{\frac{2}{p}}\rceil$, then
\eqref{eq:main} holds simultaneously for every such coefficient vector
with probability at least
\[
  1-\frac{1}{N}.
\]
The constant $C_p$ in this probability statement is independent of $N$,
the probability space, and the functions.
\end{theorem}

The cardinality in this statement is exact. The probability concerns
one selected set on which every coefficient vector satisfies the
inequality. In particular, the norm bound remains fixed as $N$ grows
and the failure probability tends to zero.

\subsection{Polynomial failure probabilities and earlier methods}
\label{sec:probability-introduction}

For a set $S\subset[N]$, define its selection constant by
\begin{equation}\label{eq:selection-constant}
  K_p(S)=\sup_{\sum_{i\in S}|a_i|^2\le1}
       \left\|\sum_{i\in S}a_i\phi_i\right\|_p,
\end{equation}
where the supremum is over complex coefficients. We put
$K_p(\varnothing)=0$. A change in the numerical weight used in the
history count gives the following extension of Theorem~\ref{thm:main}.

\begin{corollary}[Arbitrary polynomial failure probability]
\label{cor:polynomial-probability}
Let $p>2$, let $A>0$, and suppose that $N\ge2$ and the functions
$\phi_1,\ldots,\phi_N$ satisfy the hypotheses of
Theorem~\ref{thm:main}. If $S$ is chosen uniformly among the subsets of
$[N]$ of cardinality $\lceil N^{\frac{2}{p}}\rceil$, then
\begin{equation}\label{eq:polynomial-probability}
  \mathbb{P}\bigl(K_p(S)\le C_{p,A}\bigr)
  \ge1-\frac{1}{N^A},
\end{equation}
where $C_{p,A}$ depends only on $p$ and $A$.
\end{corollary}

The constant is allowed to increase with $A$; it is independent of $N$.
We prove the corollary in Section~\ref{sec:polynomial-proof} by following
the failure probability through the entire argument.

For comparison, in the real-valued setting, with the supremum in
\eqref{eq:selection-constant} taken over real coefficients, the
probabilistic estimate in Bourgain's argument is
presented in \cite[arXiv version, equation (3.16) and Theorem 3.9]{JLOV}
through a moment bound $\mathbb{E}K_p(S)^p\le B_p$, for independent
Bernoulli selection with expected cardinality $N^{\frac{2}{p}}$, where
$B_p$ depends only on $p$.
Markov's inequality yields
\[
  \mathbb{P}\bigl(K_p(S)>T\bigr)\le\frac{B_p}{T^p}
  \qquad(T>0).
\]
For fixed $T$, this displayed bound has no decay in $N$; obtaining
$\frac{1}{N}$ from it alone requires $T\ge(B_pN)^{\frac{1}{p}}$.
Theorem~\ref{thm:main} gives polynomial decay while keeping the norm
bound independent of $N$, and treats exact cardinality directly.
This comparison concerns the probability estimate obtained from that
moment bound alone.

Polynomial failure probabilities can also be recovered from
Talagrand's partition method; see \cite[Section 19.3]{TalagrandBook}.
For completeness, the relevant modification is to start the
construction in the proof of Theorem 19.3.2 at a level $n_0$ satisfying
\[
  \log N\le2^{n_0}<2\log N.
\]
For the normalized functions, the initial contribution to the weighted
partition budget is then at most $2\log N$, which is absorbed by the
existing bound of a constant multiple of $\log N$ from Theorem 19.3.4.
For any sufficiently large fixed constant $v$, the Chernoff estimate
and union bound in the proof of Theorem 19.3.3, now summed only over
$n\ge n_0$, give an exceptional probability bounded by
\[
  \sum_{n\ge n_0}\exp\left(-\frac{v2^n}{4}\right)
  \le2N^{-\frac{v}{4}},
\]
Increasing $v$ increases the resulting norm bound by a factor
independent of $N$. For an integer $1\le m\le N$, pass from Bernoulli
selection to a uniformly chosen $m$-element subset by taking selection
density $m/N$ and conditioning on cardinality $m$. That
cardinality has probability at least $\frac{1}{N+1}$, so starting with
a sufficiently large failure exponent absorbs the conditioning loss.
These are deductions from Talagrand's method, rather than the precise
probability formulation stated in his book; the contribution here is
a self-contained greedy proof that gives the quantitative estimates
and exact cardinality directly.

The mechanism is explicit: an exponential estimate controls each
discrete history, and the potential bounds the total weighted cost of
all histories used to approximate any test function. In the notation of Subsection \ref{ss:hist} below, replacing the
weight 
\[ H = \log(8NJ) \longrightarrow H_A = \log(8JN^A), \qquad A \geq 1\]
with $J=\lceil\log_2N\rceil$, reduces the failure probability to a constant
multiple of $N^{-A}$. Since this replacement changes the weight by at
most a factor depending on $A$, it changes the final norm bound by a
factor depending only on $p$ and $A$.

\subsection{The connection with \texorpdfstring{$\Lambda(p)$}{Lambda(p)} sets}

Let $\mathbb{T}=\mathbb{R}/\mathbb{Z}$ be the circle equipped with
Lebesgue measure of total mass one, and let $\mathrm{i}$ denote the
imaginary unit. The functions $x\mapsto e^{2\pi\mathrm{i}nx}$,
$n\in\mathbb{Z}$, form an orthonormal system on $\mathbb{T}$.

\begin{definition}\label{def:lambda}
For $p>2$, a set $\Lambda\subset\mathbb{Z}$ is a $\Lambda(p)$ set if
there is a finite constant $K$ such that
\begin{equation}\label{eq:lambda-definition}
  \left\|\sum_{n\in\Lambda}a_ne^{2\pi\mathrm{i}nx}
       \right\|_{L^p(\mathbb{T})}
  \le K\left(\sum_{n\in\Lambda}|a_n|^2\right)^{\frac{1}{2}}
\end{equation}
for every finitely supported family of complex coefficients
$(a_n)_{n\in\Lambda}$.
\end{definition}

For finite sets the issue is to control $K$ uniformly as the cardinality
increases. Theorem~\ref{thm:main} gives a subset of size
$\lceil N^{\frac{2}{p}}\rceil$ in every set of $N$ distinct integer frequencies,
with $K$ bounded in terms of $p$ alone, see Subsection~\ref{sec:fourier} below.

Rudin's work \cite{Rudin} established a systematic theory of
$\Lambda(p)$ sets and posed the problem of distinguishing these classes
for different exponents. Bourgain \cite{Bourgain} used the finite
selection theorem to construct, for every $p>2$, an infinite
$\Lambda(p)$ set which is not a $\Lambda(q)$ set for any $q>p$.
Talagrand \cite{Talagrand} subsequently gave an approach through the
geometry of smooth convex bodies and majorizing measures. A detailed
exposition of Bourgain's argument is given by Jung, Langowski, Ortiz,
and Vu \cite{JLOV}.

The purpose of this paper is to give a self-contained proof of the finite
selection theorem in a form that can be checked directly; its main
ingredient is a quantitative estimate for random restrictions of
coefficient sums. We prove that estimate by a finite-alphabet greedy
approximation and a countable union bound. Related greedy descent
methods for convex functions on Banach spaces are developed by
Temlyakov \cite{Temlyakov}. Here the feature needed for selection is
that the same potential controls the entire sequence of approximation
scales. The existence theorem and its optimal exponent are Bourgain's.
Our contribution lies in this greedy proof and in the explicit tracking
of its probability estimates. The proof gives the bounds in
Theorem~\ref{thm:main} and Corollary~\ref{cor:polynomial-probability},
with the comparison to the classical moment estimate described in
Section~\ref{sec:probability-introduction}.

\subsection{A common method and companion papers}

This paper is the first in a planned series developing a common approach
to selection and sampling problems for bounded orthogonal systems.
The organizing principle is to retain the state of an approximation
procedure as the accuracy changes, so that a single nonnegative potential
controls the total weighted cost of the updates over all scales.
Encoding the updates by discrete records then allows simultaneous
probability estimates to control the entire family of approximations, where -- crucially -- the geometric weights ensure that carrying an earlier update into later approximations costs only a constant multiple of its original weight.

In the present paper, this principle takes the form of greedy descent
in a smooth $L^r$ space with $r>2$. Companion manuscripts develop related
constructions using online prediction, in which a probability
distribution over elementary predictors is updated after an incorrect
prediction. Relative entropy supplies the potential in those arguments,
and records of the necessary corrections supply the discrete codes.
Thus the common ingredient is the use of one weighted update budget
throughout the approximation procedure.

The applications in these manuscripts concern restricted isometry
estimates for row restrictions of real bounded orthogonal matrices,
both through existential selection and through independent Bernoulli
sampling, and comparisons between the $L^1$ and $L^2$ norms on spans
of large subsets of bounded orthonormal systems. Restricted isometry
means that the normalized row restriction approximately preserves the
Euclidean norm of every vector with a prescribed bound on the number
of its nonzero coordinates. The comparison results in this direction
include Bourgain's restricted isometry estimate \cite{Bourgain2014}
and the subsequent estimate of Haviv and Regev \cite{HavivRegev};
for the comparison of norms on large spans, the relevant result is
Theorem~3 of Gu\'edon, Mendelson, Pajor, and Tomczak-Jaegermann
\cite{GMPT2008}. The companion manuscripts contain proposed improvements
of logarithmic factors in these estimates. Those improvements remain
under verification; their precise hypotheses, proofs, and novelty
comparisons belong to the respective papers.

Bourgain's finite $\Lambda(p)$ selection theorem provides the starting
point for the series. The proof below exhibits the common principle
directly: one greedy descent budget controls all approximation scales,
and a countable union bound gives the required random restriction
estimate. The development of this approach benefited in part from
mathematical suggestions generated in discussions with ChatGPT.
The authors take responsibility for the mathematical statements and
proofs.

\subsection{Notation and the structure of the proof}

For $1<q<\infty$, let $q'=\frac{q}{q-1}$ be the conjugate exponent,
so that $\frac{1}{q}+\frac{1}{q'}=1$. Whenever the product is
integrable, use the pairing
\[
  \langle u,v\rangle=\int_Xu(x)\overline{v(x)}\,d\mu(x).
\]
This pairing is linear in the first variable and conjugate linear in
the second. We write $\operatorname{Re}z$ and $\operatorname{Im}z$ for
the real and imaginary parts of a complex number $z$, and $1_E$ for
the function which is one on a set $E$ and
zero elsewhere.
All logarithms are natural, except for the explicitly indicated
$\log_2$. The symbols $C_r$ and $C_p$ denote positive constants
depending only on the indicated exponent; their values may increase
from one occurrence to the next.
The symbols $C_{r,A}$ and $C_{p,A}$ have the same convention, with
dependence on both indicated parameters.

The random restriction estimate is stated and proved in
Section~\ref{sec:selection}. Its deterministic part starts from a
function $g\in L^{r'}$ and approximates the finitely many numbers
$\langle g,\phi_i\rangle$. At accuracy $\varepsilon_j=2^{-j}$, let
$n_j$ be the number of new updates and let
$k_j=\sum_{h=1}^j n_h$ be the total number of updates made so far;
a nonnegative potential gives a bound on
$\sum_j n_j\varepsilon_j^2$, and since the accuracies decrease
geometrically, this also bounds $\sum_j k_j\varepsilon_j^2$.
The latter sum is precisely what appears when we add the probability
estimates over all scales.

In Section~\ref{sec:main-proof}, we apply the random restriction
estimate twice, with $r=2p$. One application gives a useful bound
for moderate values of a selected sum, while the other controls its
largest possible values. Integrating these two bounds proves
Theorem~\ref{thm:main}. Each of the two applications fails with
probability at most $\frac{1}{4N}$, and the remaining argument is
deterministic on their intersection. Section~\ref{sec:polynomial-proof}
then makes the dependence on the desired failure exponent explicit.

\section{A uniform random restriction estimate}\label{sec:selection}

The following lemma does not require orthogonality; only the
$L^r$ bounds on the functions enter its proof.

\begin{lemma}\label{lem:selection}
Let $r>2$, let $N\ge2$, and suppose that
$\phi_1,\ldots,\phi_N\in L^r(X,\mu)$ satisfy $\|\phi_i\|_r\le1$.
Fix an integer $m$ with $1\le m\le N$, and choose $S$ uniformly from
the $m$-element subsets of $[N]$. Set
\[
  \delta:=\frac{m}{N},\qquad \ell:=\log(eN),\qquad
  \beta_\delta(\lambda)
  :=\frac{\log\bigl(1+\delta(e^\lambda-1)\bigr)}{\lambda}
  \quad(\lambda>0).
\]
For every fixed $\lambda$ with $1\le\lambda\le\ell$, there is an
event of probability at least $1-\frac{1}{4N}$ on which
\begin{equation}\label{eq:selection}
\begin{split}
  \sum_{i\in S}|\langle g,\phi_i\rangle|^2
  \le C_r\biggl(&\frac{\ell}{\lambda}\|g\|_{r'}^2 +\beta_\delta(\lambda)
       \sum_{i=1}^N|\langle g,\phi_i\rangle|^2\biggr)
\end{split}
\end{equation}
simultaneously for every $g\in L^{r'}(X,\mu)$.
The constant $C_r$ is independent of $N,m,\lambda$, and the functions.
\end{lemma}

The coefficient in this formulation satisfies
\begin{equation}\label{eq:beta-upper}
  0<\beta_\delta(\lambda)
  \le\min\left\{1,\frac{\delta(e^\lambda-1)}{\lambda}\right\}.
\end{equation}
Indeed, $1+\delta(e^\lambda-1)\le e^\lambda$ because $\delta\le1$,
and $\log(1+u)\le u$ for $u\ge0$. Thus the lemma also implies
the simpler estimate obtained by replacing $\beta_\delta(\lambda)$
with $\frac{\delta(e^\lambda-1)}{\lambda}$. Retaining the logarithm
avoids an unnecessary loss when the exponential factor is large.

The word ``simultaneously'' is essential: after choosing a set in the
event from the lemma, the test function $g$ may depend on that set.
The parameter $\lambda$, however, is fixed before the random choice.
We will later intersect the events for two specified values of
$\lambda$.

\subsection{The smoothness inequality}

Fix $r>2$. For $v\in L^r$, define
\[
  \Psi(v)=\frac{1}{2}\|v\|_r^2,
  \qquad
  G(v)=
  \begin{cases}
    \|v\|_r^{2-r}|v|^{r-2}v,&v\ne0,\\
    0,&v=0,
  \end{cases}
\]
where $v=0$ means equality to zero almost everywhere. A direct
calculation gives $G(v)\in L^{r'}$ and
\begin{equation}\label{eq:duality-map}
  \|G(v)\|_{r'}=\|v\|_r,
  \qquad
  \operatorname{Re}\langle G(v),v\rangle=\|v\|_r^2.
\end{equation}
We regard complex $L^r$ as a real vector space when differentiating
real-valued functions. In particular,
\[
  \left.\frac{d}{dt}\Psi(v+th)\right|_{t=0}
  =\operatorname{Re}\langle G(v),h\rangle
  \qquad(t\in\mathbb{R}).
\]
We begin with the following technical lemma, which follows from direct computation; for completeness we provide the details.
\begin{lemma}\label{lem:smoothness}
For every $v,h\in L^r$,
\begin{equation}\label{eq:smoothness}
  \Psi(v+h)
  \le\Psi(v)+\operatorname{Re}\langle G(v),h\rangle
       +\frac{r-1}{2}\|h\|_r^2.
\end{equation}
\end{lemma}

\begin{proof}
For real $t$, put $w(t)=v+th$ and
$A(t)=\int_X|w(t)|^r\,d\mu$. Differentiation under the integral gives
\[
  A'(t)=r\int_X|w(t)|^{r-2}
             \operatorname{Re}(\overline{w(t)}h)\,d\mu.
\]
The pointwise second derivative of $|w(t)|^r$ is
\[
  r|w(t)|^{r-2}|h|^2
  +r(r-2)|w(t)|^{r-4}
       \bigl(\operatorname{Re}(\overline{w(t)}h)\bigr)^2,
\]
and at points where $w(t)=0$, the second term is interpreted as zero;
this is its continuous extension, since $r>2$. The inequality
$|\operatorname{Re}(\overline{w(t)}h)|\le |w(t)||h|$ therefore yields the upper estimate
\begin{equation}\label{eq:A-second}
  A''(t)\le r(r-1)\int_X|w(t)|^{r-2}|h|^2\,d\mu.
\end{equation}
These differentiations are justified by dominated convergence on
every bounded interval of $t$: the second derivatives are bounded by
a constant times $(|v|+|h|)^{r-2}|h|^2$, which is integrable by
H\"older's inequality.

Whenever $A(t)>0$, the chain rule yields
\begin{align*}
  \frac{d^2}{dt^2}\Psi(w(t))
  & =\frac{2-r}{r^2}A(t)^{\frac{2}{r}-2}(A'(t))^2
      +\frac{1}{r}A(t)^{\frac{2}{r}-1}A''(t)\\
  &\le(r-1)\|w(t)\|_r^{2-r}
       \int_X|w(t)|^{r-2}|h|^2\,d\mu\\
  &\le(r-1)\|h\|_r^2;
\end{align*}
to verify these inequalities, note that since $r > 2$, the first term in the first line is nonpositive, while the last line uses
\[
  \int_X|w(t)|^{r-2}|h|^2\,d\mu
  \le\|w(t)\|_r^{r-2}\|h\|_r^2.
\]
If $w(t_0)=0$ in $L^r$ for some $t_0$, then
$w(t)=(t-t_0)h$ for all $t$, and
$\Psi(w(t))=\frac{1}{2}(t-t_0)^2\|h\|_r^2$,
so the same second derivative bound holds also in this case. Taylor’s formula with integral remainder proves \eqref{eq:smoothness}.
\end{proof}

\subsection{A greedy approximation retained across all scales}

Fix $g\in L^{r'}$ with $\|g\|_{r'}\le1$, and set
\[
  t_i=\langle g,\phi_i\rangle\qquad(i\in[N]).
\]
H\"older's inequality gives $|t_i|\le1$. Let
\[
  J=\lceil\log_2N\rceil,
  \qquad\varepsilon_j=2^{-j}\quad(1\le j\le J),
\]
so that in particular, $\varepsilon_J\le\frac{1}{N}$.
The approximation uses the four-element set
\[
  \Omega=\{1,-1,\mathrm{i},-\mathrm{i}\}.
\]
For every complex number $z$, one can choose $\omega\in\Omega$ such that
\begin{equation}\label{eq:phase}
  \operatorname{Re}(\overline{\omega}z)
  \ge\frac{|z|}{\sqrt{2}}:
\end{equation}
the largest of $|\operatorname{Re}z|$ and
$|\operatorname{Im}z|$ is at least $\frac{|z|}{\sqrt{2}}$, and the
four choices of $\omega$ select either signed component.

Start with $v=0$. At scale $j$, check whether there is an index $i$
for which
\[
  |t_i-\langle G(v),\phi_i\rangle|>\varepsilon_j.
\]
If there is, take the smallest such index, put
$z=t_i-\langle G(v),\phi_i\rangle$, and choose $\omega$ as in
\eqref{eq:phase}, and replace $v$ by
\begin{equation}\label{eq:update}
  v+\frac{\varepsilon_j}{\sqrt{2}(r-1)}\omega\phi_i.
\end{equation}
Fix an ordering of $\Omega$ to resolve any ties, and continue at scale $j$
until every coordinate error is at most $\varepsilon_j$. Then proceed
to scale $j+1$, starting from the current value of $v$.

To prove termination and control all the updates, define
\[
  P_g(v)=\frac{1}{2}\|v\|_r^2
          -\operatorname{Re}\langle g,v\rangle+\frac{1}{2}.
\]
Since $\|g\|_{r'}\le1$,
\begin{equation}\label{eq:potential-positive}
  P_g(v)\ge\frac{1}{2}\|v\|_r^2-\|v\|_r+\frac{1}{2}
     =\frac{1}{2}(\|v\|_r-1)^2\ge0,
  \qquad P_g(0)=\frac{1}{2}.
\end{equation}
And, for an update at scale $j$, with
$\eta=\frac{\varepsilon_j}{\sqrt{2}(r-1)}$, by
Lemma~\ref{lem:smoothness},
\begin{align*}
  P_g(v+\eta\omega\phi_i)-P_g(v)
  &\le -\eta\operatorname{Re}
      \bigl(\overline{\omega}\langle g-G(v),\phi_i\rangle\bigr)
       +\frac{r-1}{2}\eta^2\|\phi_i\|_r^2\\
  &\le-\frac{\eta}{\sqrt{2}}
        |t_i-\langle G(v),\phi_i\rangle|
       +\frac{r-1}{2}\eta^2\\
  &<-\frac{\varepsilon_j^2}{4(r-1)}.
\end{align*}
In particular, an infinite number of updates at any one scale would force the
potential to become negative, so every scale terminates.

Let $n_j$ be the number of updates performed at scale $j$, let $v_j$
be the state when that scale is completed, and let
$k_j=\sum_{h=1}^j n_h$. The construction gives
\begin{equation}\label{eq:approximation}
  \max_{i\in[N]}|t_i-\langle G(v_j),\phi_i\rangle|
  \le\varepsilon_j.
\end{equation}
More significantly, adding all potential decreases and using
\eqref{eq:potential-positive} gives
\begin{equation}\label{eq:budget}
  \sum_{j=1}^Jn_j\varepsilon_j^2\le2(r-1).
\end{equation}
There is just one initial potential of size $\frac{1}{2}$ in this
calculation, because the state is retained between successive scales.

As alluded to above, the count of a history at scale $j$ will involve $k_j$, not merely
$n_j$, but the geometric decrease of the accuracies gives the needed
conversion:
\begin{align}
  \sum_{j=1}^J\varepsilon_j^2k_j
  &=\sum_{h=1}^Jn_h\sum_{j=h}^J\varepsilon_j^2\notag\\
  &\le\frac{4}{3}\sum_{h=1}^Jn_h\varepsilon_h^2
   \le\frac{8}{3}(r-1).\label{eq:cumulative}
\end{align}
Since $\sum_{j=1}^{\infty}\varepsilon_j^2=\frac{1}{3}$, we also have
\begin{equation}\label{eq:cumulative-budget}
  \sum_{j=1}^J\varepsilon_j^2(k_j+1)
  \le\frac{8r-7}{3};
\end{equation}
this is the bound that will absorb the entire cost of the histories.

\subsection{An exponential estimate for sampling without replacement}

Fix a deterministic set $I\subset[N]$. For the random $m$-element
set $S$ and any $\lambda>0$, we claim that
\begin{equation}\label{eq:mgf}
  \mathbb{E}\exp(\lambda|S\cap I|)
  \le\bigl(1+\delta(e^\lambda-1)\bigr)^{|I|}
  =\exp\bigl(\lambda\beta_\delta(\lambda)|I|\bigr),
  \qquad\delta=\frac{m}{N}.
\end{equation}
To see this without assuming independence, expand the product:
\begin{align*}
  \mathbb{E}\exp(\lambda|S\cap I|)
  &=\mathbb{E}\prod_{i\in I}
       \bigl(1+(e^\lambda-1)1_{\{i\in S\}}\bigr)\\
  &=\sum_{A\subset I}(e^\lambda-1)^{|A|}\mathbb{P}(A\subset S).
\end{align*}
For $a=|A|\le m$, counting the subsets containing $A$ gives
\[
  \mathbb{P}(A\subset S)
  =\frac{\binom{N-a}{m-a}}{\binom{N}{m}}
  =\prod_{h=0}^{a-1}\frac{m-h}{N-h}
  \le\left(\frac{m}{N}\right)^a.
\]
The empty product is one, and the probability is zero when $a>m$.
It follows that
\[
  \mathbb{E}\exp(\lambda|S\cap I|)
  \le\bigl(1+\delta(e^\lambda-1)\bigr)^{|I|}
  =\exp\bigl(\lambda\beta_\delta(\lambda)|I|\bigr),
\]
which proves \eqref{eq:mgf}. In particular, Markov's inequality gives,
for every $u\ge0$,
\begin{equation}\label{eq:fixed-set-tail}
  \mathbb{P}\left(
    \lambda|S\cap I|>\lambda\beta_\delta(\lambda)|I|+u\right)
  \le e^{-u}.
\end{equation}

\subsection{Counting every possible history}\label{ss:hist}

An update in \eqref{eq:update} is specified by three discrete choices:
its scale $j\in\{1,\ldots,J\}$, its index $i\in[N]$, and its phase
$\omega\in\Omega$. A history of length $k$ is any ordered list
\[
  \mathcal{H}=((j_1,i_1,\omega_1),\ldots,(j_k,i_k,\omega_k))
\]
of such triples. We include all these lists, even those which cannot
arise from the approximation procedure, so in particular there are exactly
$(4NJ)^k$ lists of length $k$. Define their states by
\[
  v(\mathcal{H})
  =\sum_{a=1}^k
      \frac{\varepsilon_{j_a}}{\sqrt{2}(r-1)}\omega_a\phi_{i_a},
\]
with the convention that the empty history has length zero and state zero. The state is
determined completely by the list; no continuously
varying coefficient needs to be encoded.

For every history $\mathcal{H}$ and every $j\in\{1,\ldots,J\}$, define
the deterministic set
\[
  I(\mathcal{H},j)
  =\left\{i\in[N]:
        |\langle G(v(\mathcal{H})),\phi_i\rangle|
             >2\varepsilon_j\right\}.
\]
Put $H=\log(8NJ)$ and apply \eqref{eq:fixed-set-tail} to this set with
$u=(k+1)H$, where $k$ is the length of $\mathcal{H}$.
Since there are $J$ scales to test, the probability that any history and scale violate the resulting
inequality is at most
\begin{align}\label{eq:history-failure}
  J\sum_{k=0}^{\infty}(4NJ)^k e^{-(k+1)H}
  &=\frac{Je^{-H}}{1-4NJ e^{-H}}=\frac{1}{4N}.
\end{align}
Thus there is an event of probability at least
$1-\frac{1}{4N}$ on which
\begin{equation}\label{eq:uniform-histories}
\begin{split}
  |S\cap I(\mathcal{H},j)|
  \le{}&\beta_\delta(\lambda)
                     |I(\mathcal{H},j)|+\frac{(k+1)H}{\lambda}
\end{split}
\end{equation}
for every history of length $k$ and every scale $j$; we will only apply this when $1 \leq \lambda \leq \ell$.

This event was defined using the functions $\phi_i$ and the countable
collection of all finite histories; it was not defined using a target
function $g$ so it can be fixed before any target is chosen. This
also avoids any measurability question involving an uncountable
intersection over test functions.

\subsection{From histories to squared coefficients}

Fix a set $S$ in the event from \eqref{eq:uniform-histories},
take any $g\in L^{r'}$ with $\|g\|_{r'}\le1$, and perform the
approximation procedure. Retain its notation $t_i,v_j,k_j$ and define
\[
  I_j=\{i\in[N]:|\langle G(v_j),\phi_i\rangle|>2\varepsilon_j\};
\]
the approximation estimate \eqref{eq:approximation} gives
\begin{equation}\label{eq:inclusions}
  \{i\in[N]:|t_i|>3\varepsilon_j\}
  \subset I_j
  \subset\{i\in[N]:|t_i|>\varepsilon_j\}
\end{equation}
where both inclusions follow from the triangle inequality. The first ensures
that the sets $I_j$ detect large true coefficients; the second prevents
them from counting coefficients that are too small.

For a real number $u$ with $0\le u\le1$, the dyadic estimate
\begin{equation}\label{eq:dyadic}
  u^2\le9\varepsilon_J^2
       +36\sum_{j=1}^J\varepsilon_j^2
                          1_{\{u>3\varepsilon_j\}}
\end{equation}
holds, as we quickly verify: if $u\le3\varepsilon_J$, the first term suffices; otherwise, let $j$ be the first index such that $u>3\varepsilon_j$, noting that $j \geq 2$ since $u\le1<3\varepsilon_1$; by minimality
$u\le3\varepsilon_{j-1}=6\varepsilon_j$.

If we apply \eqref{eq:dyadic} to each $|t_i|$, $i\in S$, and use
\eqref{eq:inclusions}, then since $|S|\le N$ and
$\varepsilon_J\le\frac{1}{N}$,
\[
  \sum_{i\in S}|t_i|^2
  \le\frac{9}{N}
       +36\sum_{j=1}^J\varepsilon_j^2|S\cap I_j|.
\]
The state $v_j$ is represented by its actual cumulative history,
whose length is $k_j$, so \eqref{eq:uniform-histories} yields the bound
\begin{equation}\label{eq:sum-intermediate}
\begin{split}
  \sum_{i\in S}|t_i|^2
  \le{}&\frac{9}{N}
      +36\beta_\delta(\lambda)
         \sum_{j=1}^J\varepsilon_j^2|I_j|\\
      &+\frac{36H}{\lambda}
         \sum_{j=1}^J\varepsilon_j^2(k_j+1).
\end{split}
\end{equation}
The last sum is bounded by \eqref{eq:cumulative-budget}, while for the other
sum, the second inclusion in \eqref{eq:inclusions} and the geometric decay of $\{ \varepsilon_j \}$ imply
\begin{align}
  \sum_{j=1}^J\varepsilon_j^2|I_j|
  &\le\sum_{i=1}^N
         \sum_{\{j:\,\varepsilon_j<|t_i|\}}\varepsilon_j^2\notag\\
  &\le\frac{4}{3}\sum_{i=1}^N|t_i|^2.\label{eq:mean-levels}
\end{align}

Combining \eqref{eq:sum-intermediate}, \eqref{eq:cumulative-budget},
and \eqref{eq:mean-levels} yields
\begin{equation}\label{eq:normalized-selection}
\begin{split}
  \sum_{i\in S}|t_i|^2
  \le{}&\frac{9}{N}
     +48\beta_\delta(\lambda)
                         \sum_{i=1}^N|t_i|^2\\
     &+12(8r-7)\frac{H}{\lambda}.
\end{split}
\end{equation}
For $N\ge2$, we have $J\le N$, and hence
\[
  H=\log(8NJ)\le3\log(eN)=3\ell.
\]
Since we are only interested in the case where $\frac{\ell}{\lambda}\ge1$, we can consequently absorb
$\frac{9}{N}$ into a constant times $\frac{\ell}{\lambda}$.
This is the only point where the upper restriction $\lambda\le\ell$
is used; it suffices for both applications below.
Thus \eqref{eq:normalized-selection} proves
\[
  \sum_{i\in S}|\langle g,\phi_i\rangle|^2
  \le C_r\left(\frac{\ell}{\lambda}
       +\beta_\delta(\lambda)
            \sum_{i=1}^N|\langle g,\phi_i\rangle|^2\right)
\]
for every $g$ of $L^{r'}$ norm at most one. If $g\ne0$, apply this
inequality to $\frac{g}{\|g\|_{r'}}$ and multiply by
$\|g\|_{r'}^2$. The case $g=0$ is immediate. This proves
\eqref{eq:selection} for all $g\in L^{r'}$ on the same event and
completes the proof of Lemma~\ref{lem:selection}.

\begin{remark}\label{rem:budget}
The estimate \eqref{eq:cumulative} is stronger than a separate bound
$k_j\le C_r\varepsilon_j^{-2}$ for each $j$. Those separate bounds
would give only
$\sum_j\varepsilon_j^2k_j\le C_rJ$. In the proof above, the potential
decreases throughout the whole procedure, and each update at scale
$h$ is paid for in the geometric sum via
$\sum_{j=h}^J\varepsilon_j^2\le\frac{4}{3}\varepsilon_h^2$;
this explains why no factor $J$ appears in
\eqref{eq:cumulative-budget}.
\end{remark}

\section{Proof of the selection theorem}\label{sec:main-proof}

We now prove Theorem~\ref{thm:main}. The case $N=1$ follows from
$\|\phi_1\|_p\le1$, so assume $N\ge2$. Set
\[
  r:=2p,\qquad
  \alpha:=1-\frac{2}{p}=\frac{p-2}{p},\qquad
  m:=\left\lceil N^{\frac{2}{p}}\right\rceil,
  \qquad\ell:=\log(eN).
\]
Thus $0<\alpha<1$, $1\le m\le N$, and
\begin{equation}\label{eq:density}
  \delta=\frac{m}{N}\le2N^{-\alpha};
\end{equation}
note that since $\mu$ is a probability measure,
$\|\phi_i\|_r\le\|\phi_i\|_\infty\le1$.

\subsection{Two estimates on one selected set}

Orthogonality gives the bound
\begin{equation}\label{eq:bessel}
  \sum_{i=1}^N|\langle g,\phi_i\rangle|^2\le\|g\|_2^2
  \qquad(g\in L^2).
\end{equation}
For completeness, discard any zero functions and put
$\psi_i=\frac{\phi_i}{\|\phi_i\|_2}$ for the remaining indices, so that $\{ \psi_i \}$ are orthonormal. Bessel's inequality and
$\|\phi_i\|_2\le1$ give
\[
  \sum_i|\langle g,\phi_i\rangle|^2
  =\sum_i\|\phi_i\|_2^2|\langle g,\psi_i\rangle|^2
  \le\sum_i|\langle g,\psi_i\rangle|^2
  \le\|g\|_2^2.
\]

Choose $S$ uniformly among the $m$-element subsets of $[N]$ and
apply Lemma~\ref{lem:selection} with the two fixed parameters
\[
  \lambda_1=1,
  \qquad\lambda_2=1+\frac{\alpha}{2}\log N,
\]
both of which belong to $[1,\ell]$. The union bound shows that the two
conclusions hold on the same set with probability at least
\begin{equation}\label{eq:two-event-probability}
  1-\frac{1}{4N}-\frac{1}{4N}
  =1-\frac{1}{2N}\ge1-\frac{1}{N}>0.
\end{equation}
Fix a set $S$ in this intersection. The two events need not be
independent. Each event gives its estimate for every test function,
and every remaining step of the proof is deterministic on their
intersection; no further probability loss occurs when the coefficients
or test functions are chosen.

Using \eqref{eq:beta-upper}, \eqref{eq:density}, and
\eqref{eq:bessel}, the first application
gives
\begin{equation}\label{eq:first-application}
  \sum_{i\in S}|\langle g,\phi_i\rangle|^2
  \le C_p\bigl(\ell\|g\|_{r'}^2
                   +N^{-\alpha}\|g\|_2^2\bigr)
  \qquad(g\in L^2).
\end{equation}
For the second application, observe that
\[
  \frac{\ell}{\lambda_2}\le\frac{2}{\alpha},
  \qquad
  \beta_\delta(\lambda_2)
  \le\frac{\delta(e^{\lambda_2}-1)}{\lambda_2}
  \le\frac{4e}{\alpha\ell}N^{-\frac{\alpha}{2}}.
\]
Here $\lambda_2\ge\frac{\alpha}{2}\ell$ because $\alpha<1$,
while \eqref{eq:density} gives
$\delta e^{\lambda_2}\le2eN^{-\frac{\alpha}{2}}$.
We obtain
\begin{equation}\label{eq:second-application}
  \sum_{i\in S}|\langle g,\phi_i\rangle|^2
  \le C_p\left(\|g\|_{r'}^2
                   +\frac{N^{-\frac{\alpha}{2}}}{\ell}
                       \|g\|_2^2\right)
  \qquad(g\in L^2).
\end{equation}
Since $r'<2$ and $\mu(X)=1$, every $L^2$ function belongs to
$L^{r'}$, so these uses of Lemma~\ref{lem:selection} are legitimate.

The parameter choice balances the two coefficients in
\eqref{eq:selection}: taking $\lambda=1$ keeps the coefficient of
the full coefficient sum bounded by a constant times the sampling
density, while taking $\lambda=\lambda_2$ makes
$\frac{\ell}{\lambda}$ bounded in terms of $p$ alone. Although
$e^{\lambda_2}$ grows with $N$, the product
$\delta e^{\lambda_2}$ still decreases as a negative power of $N$.
The factor $\frac{1}{\ell}$ retained in
\eqref{eq:second-application} is a further gain. These choices are
made to obtain the optimal cardinality exponent; no optimization of
the numerical constant $C_p$ is asserted, and each parameter is fixed
before sampling, as required by Lemma~\ref{lem:selection}.

\subsection{Distribution bounds for a selected linear combination}

Fix coefficients with $\sum_{i\in S}|a_i|^2\le1$, and put
\[
  f=\sum_{i\in S}a_i\phi_i,
  \qquad M=\sqrt{m};
\]
the bound on the functions and the Cauchy--Schwarz inequality imply
\begin{equation}\label{eq:maximum}
  \|f\|_\infty\le\sum_{i\in S}|a_i|
      \le\sqrt{m}=M
      \le\sqrt{2}\,N^{\frac{1}{p}}.
\end{equation}
For $t>0$, define
\[
  E_t=\{x\in X:|f(x)|>t\},\qquad F(t)=\mu(E_t),
\]
and define the test function
\begin{align}\label{e:gt}
  g_t(x)=
  \begin{cases}
    \frac{f(x)}{|f(x)|},&x\in E_t,\\
    0,&x\notin E_t,
  \end{cases}
\end{align}
where the first case is well-defined since in this case $|f(x)| > t>0$.
The function $g_t$ lies in $L^2$, and
\begin{equation}\label{eq:test-norms}
  \|g_t\|_{r'}^2=F(t)^{\frac{2}{r'}},
  \qquad\|g_t\|_2^2=F(t);
\end{equation}
the conjugations in the pairing yield the distributional estimate
\begin{align}
  tF(t)
  &\le\int_{E_t}|f|\,d\mu\notag\\
  &=\left|\int_X f\overline{g_t}\,d\mu\right|
   =\left|\sum_{i\in S}a_i
               \overline{\langle g_t,\phi_i\rangle}\right|\notag\\
  &\le\left(\sum_{i\in S}|\langle g_t,\phi_i\rangle|^2
                                      \right)^{\frac{1}{2}}.
                                      \label{eq:duality-tail}
\end{align}

We now use a simple algebraic observation: if $A,B,F\ge0$ and $t>0$
satisfy
\[
  t^2F^2\le AF^{\frac{2}{r'}}+BF,
\]
then
\begin{equation}\label{eq:tail-algebra}
  F\le(2A)^{\frac{r}{2}}t^{-r}+2Bt^{-2}.
\end{equation}
When $F=0$, there is nothing to prove. On the other hand, when $F>0$, at least one of the two terms on the right of the preceding inequality is at least
$\frac{1}{2}t^2F^2$. In the first case,
\[
  F^{\frac{2}{r}}
  =F^{2-\frac{2}{r'}}\le2At^{-2},
\]
which gives the first term of \eqref{eq:tail-algebra}; in the second
case, division by $F$ gives $F\le2Bt^{-2}$.

Square \eqref{eq:duality-tail}, apply
\eqref{eq:first-application} and \eqref{eq:second-application},
and use \eqref{eq:test-norms}: since $r=2p$, the observation yields
the two bounds
\begin{equation}\label{eq:first-tail}
  F(t)\le C_p\bigl(\ell^pt^{-2p}+N^{-\alpha}t^{-2}\bigr)
  \qquad(t>0)
\end{equation}
and
\begin{equation}\label{eq:second-tail}
  F(t)\le C_p\left(t^{-2p}
          +\frac{N^{-\frac{\alpha}{2}}}{\ell}t^{-2}\right)
  \qquad(t>0).
\end{equation}
Although $g_t$ depends on $f$ and on $S$, both applications are valid:
the estimates were established for every test function on the same
fixed set $S$.

\subsection{Integrating the two distribution bounds}

Tonelli's theorem applied to the identity
$|f(x)|^p=\int_0^{|f(x)|}pt^{p-1}\,dt$ gives
\begin{equation}\label{eq:layer-cake}
  \|f\|_p^p=p\int_0^M t^{p-1}F(t)\,dt.
\end{equation}
Here we used \eqref{eq:maximum}, which implies $F(t)=0$ for $t\ge M$.
We use $F(t)\le1$ for $0<t<1$, use
\eqref{eq:second-tail} for $1\le t\le\ell$, and use
\eqref{eq:first-tail} for $t\ge\ell$.

Extending nonnegative integrals gives
\begin{equation}\label{eq:integral-split}
\begin{split}
  \|f\|_p^p\le C_p\biggl(&1
    +\int_1^\ell t^{-p-1}\,dt
    +\frac{N^{-\frac{\alpha}{2}}}{\ell}
          \int_1^\ell t^{p-3}\,dt\\
    &+\ell^p\int_\ell^\infty t^{-p-1}\,dt
    +N^{-\alpha}\int_0^M t^{p-3}\,dt\biggr).
\end{split}
\end{equation}
Note that this estimate also covers $M\le\ell$ (in that case the actual
contribution from $t\ge\ell$ is zero, and the added integrals are
nonnegative); the last integral is finite because $p>2$.

Consequently,
\begin{equation}\label{eq:integrated}
  \|f\|_p^p
  \le C_p\left(1+N^{-\frac{\alpha}{2}}\ell^{p-3}
                   +N^{-\alpha}M^{p-2}\right).
\end{equation}
Both remaining quantities are bounded independently of $N$.
For an explicit check of the logarithmic term, set $x=\log N\ge0$, and note that since $\ell\ge1$, we have
\[
  N^{-\frac{\alpha}{2}}\ell^{p-3}
  \le N^{-\frac{\alpha}{2}}\ell^{p-2}
  =\left((1+x)e^{-\frac{x}{2p}}\right)^{p-2} \leq C_p
\]
since the function $(1+x)e^{-\frac{x}{2p}}$ on $[0,\infty)$ is maximized at
$x=2p-1$.
For the last term, \eqref{eq:maximum} gives
\[
  N^{-\alpha}M^{p-2}
  \le2^{\frac{p-2}{2}}
          N^{-\alpha+\frac{p-2}{p}}
  =2^{\frac{p-2}{2}}.
\]
We have proved $\|f\|_p\le C_p$ whenever
$\sum_{i\in S}|a_i|^2\le1$.
Rescaling a nonzero coefficient vector proves \eqref{eq:main}; the
zero vector is immediate. The set $S$ was fixed before the coefficients
were chosen, so the inequality holds for all coefficient vectors on
that set. Since this conclusion holds for every $S$ in the event
from \eqref{eq:two-event-probability}, we have also proved
\[
  \mathbb{P}\bigl(K_p(S)\le C_p\bigr)
  \ge1-\frac{1}{2N}\ge1-\frac{1}{N}.
\]
This proves both assertions of Theorem~\ref{thm:main}.

\begin{remark}
The two uses of Lemma~\ref{lem:selection} have distinct purposes.
In \eqref{eq:second-tail}, the coefficient of $t^{-2p}$ is independent
of $N$, which controls the integral up to $\ell$.
In \eqref{eq:first-tail}, the smaller coefficient $N^{-\alpha}$ of
$t^{-2}$ is exactly what is needed at the largest amplitude
$M\le\sqrt{2}N^{\frac{1}{p}}$. Above $\ell$, the factor $\ell^p$ in the
other term is canceled by integrating $t^{-p-1}$ from $\ell$ to
infinity. This is why the argument reaches the exponent $\frac{2}{p}$.
\end{remark}

\subsection{Proof of the polynomial probability bound}
\label{sec:polynomial-proof}

\begin{proof}[Proof of Corollary~\ref{cor:polynomial-probability}]
First suppose that $A\ge1$. Keep the deterministic approximation
procedure unchanged, and put
\[
  H_A=\log(8JN^A),\qquad J=\lceil\log_2N\rceil.
\]
For a fixed $\lambda\in[1,\ell]$, apply
\eqref{eq:fixed-set-tail} to every set $I(\mathcal{H},j)$ with
$u=(k+1)H_A$, where $k$ is the length of the history $\mathcal{H}$.
The union bound over all histories and scales is
\begin{align}
  J\sum_{k=0}^{\infty}(4NJ)^k e^{-(k+1)H_A}
  &=\frac{Je^{-H_A}}{1-4NJ e^{-H_A}}\notag\\
  &=\frac{\frac{1}{8N^A}}
          {1-\frac{1}{2}N^{1-A}}
    \le\frac{1}{4N^A}.
    \label{eq:polynomial-history-failure}
\end{align}
Here $N^{1-A}\le1$, so the geometric series converges and its
denominator is at least $\frac{1}{2}$. Outside an event of probability
at most $\frac{1}{4N^A}$, we therefore have
\[
  |S\cap I(\mathcal{H},j)|
  \le\beta_\delta(\lambda)|I(\mathcal{H},j)|
       +\frac{(k+1)H_A}{\lambda}
\]
simultaneously for every history and scale.

The passage from histories to squared coefficients uses only this
inequality and the deterministic estimates
\eqref{eq:cumulative-budget} and \eqref{eq:mean-levels}. Thus
\eqref{eq:normalized-selection} becomes, for every
$g\in L^{r'}$ with $\|g\|_{r'}\le1$,
\[
  \sum_{i\in S}|\langle g,\phi_i\rangle|^2
  \le\frac{9}{N}
       +48\beta_\delta(\lambda)
           \sum_{i=1}^N|\langle g,\phi_i\rangle|^2
       +12(8r-7)\frac{H_A}{\lambda}.
\]
Since $J\le N$ and $A\ge1$,
\[
  H_A\le\log 8+(A+1)\log N
       \le(A+2)\log(eN)=(A+2)\ell.
\]
Using $\lambda\le\ell$ to absorb $\frac{9}{N}$, and then rescaling
$g$, proves
\begin{equation}\label{eq:selection-polynomial}
\begin{split}
  \sum_{i\in S}|\langle g,\phi_i\rangle|^2
  \le C_{r,A}\biggl(&\frac{\ell}{\lambda}\|g\|_{r'}^2 +\beta_\delta(\lambda)
           \sum_{i=1}^N|\langle g,\phi_i\rangle|^2\biggr)
\end{split}
\end{equation}
simultaneously for all $g\in L^{r'}$, on an event of probability at
least $1-\frac{1}{4N^A}$. The constant $C_{r,A}$ is independent of
$N$, $m$, $\lambda$, and the functions.

Now set $r=2p$, $\alpha=1-\frac{2}{p}$, and
$m=\lceil N^{\frac{2}{p}}\rceil$, as in the proof of
Theorem~\ref{thm:main}. Apply \eqref{eq:selection-polynomial} with
the same two fixed parameters
\[
  \lambda_1=1,\qquad
  \lambda_2=1+\frac{\alpha}{2}\log N.
\]
Their events intersect with probability at least
$1-\frac{1}{2N^A}$. On this intersection, the calculations leading to
\eqref{eq:first-application} and \eqref{eq:second-application} give,
for every $g\in L^2$,
\[
  \sum_{i\in S}|\langle g,\phi_i\rangle|^2
  \le C_{p,A}\bigl(\ell\|g\|_{r'}^2
                         +N^{-\alpha}\|g\|_2^2\bigr)
\]
and
\[
  \sum_{i\in S}|\langle g,\phi_i\rangle|^2
  \le C_{p,A}\left(\|g\|_{r'}^2
                  +\frac{N^{-\frac{\alpha}{2}}}{\ell}
                       \|g\|_2^2\right).
\]
For any coefficient vector with $\sum_{i\in S}|a_i|^2\le1$, define
$f$, $M$, and $F(t)$ as before. The same test function $g_t$ and
\eqref{eq:tail-algebra} yield
\[
  F(t)\le C_{p,A}
            \bigl(\ell^pt^{-2p}+N^{-\alpha}t^{-2}\bigr)
\]
and
\[
  F(t)\le C_{p,A}
            \left(t^{-2p}
               +\frac{N^{-\frac{\alpha}{2}}}{\ell}t^{-2}\right).
\]
Integrating over the same ranges as in
\eqref{eq:integral-split} gives
\[
  \|f\|_p^p
  \le C_{p,A}\left(
        1+N^{-\frac{\alpha}{2}}\ell^{p-3}
          +N^{-\alpha}M^{p-2}\right)
  \le C_{p,A}.
\]
The last inequality follows from the explicit bounds after
\eqref{eq:integrated}, which depend only on $p$. Taking the
$p$th root and adjusting $C_{p,A}$ proves $K_p(S)\le C_{p,A}$ on
the same event. Hence
\[
  \mathbb{P}\bigl(K_p(S)\le C_{p,A}\bigr)
  \ge1-\frac{1}{2N^A}\ge1-\frac{1}{N^A}.
\]
All coefficient vectors are handled on this one event.

If $0<A<1$, the assertion follows from Theorem~\ref{thm:main},
because $N^{-1}\le N^{-A}$. This proves the corollary for every
$A>0$.
\end{proof}

\section{Complements}\label{sec:complements}
\subsection{Fourier selection and the sharp exponent}\label{sec:fourier}
The following corollary is the specialization of
Theorem~\ref{thm:main} and Corollary~\ref{cor:polynomial-probability}
to the orthonormal characters on the torus.
\begin{corollary}\label{cor:fourier}
Let $p>2$, and let $A\subset\mathbb{Z}$ be a finite set of
cardinality $N\ge1$. There is a subset $\Lambda\subset A$ of
cardinality $\lceil N^{\frac{2}{p}}\rceil$ such that
\[
  \left\|\sum_{n\in\Lambda}a_ne^{2\pi\mathrm{i}nx}
       \right\|_{L^p(\mathbb{T})}
  \le C_p\left(\sum_{n\in\Lambda}|a_n|^2\right)^{\frac{1}{2}}
\]
for all complex coefficients, with $C_p$ independent of $A$ and $N$.
If $N\ge2$, a uniformly chosen subset $\Lambda\subset A$ of this
cardinality satisfies the inequality simultaneously for all
coefficients with probability at least $1-\frac{1}{N}$. More
generally, for every fixed $b>0$, the success probability is at least
$1-\frac{1}{N^b}$ after replacing $C_p$ by a constant $C_{p,b}$.
\end{corollary}

The cardinality exponent is sharp in the worst case, already for
$A=[N]$, by the standard estimate near the origin for trigonometric
sums with equal coefficients; see \cite{Bourgain,Rudin}.
\subsection{Selection under \texorpdfstring{$L^r$}{Lr} normalization}
\label{sec:lr-normalization}

The pointwise bound in \eqref{eq:maximum} can be replaced by a
distributional estimate under $L^r$ normalization. The comparison with
earlier probability estimates is discussed in
Section~\ref{sec:probability-introduction}.

Agaev \cite[Theorem 1]{Agaev} established sharp selection for orthonormal systems uniformly bounded in $L^r$ when $p>2$ is an even integer and $p<r\le2p-2$. The obstruction determining this exponent
goes back to Gaposhkin, as reproduced in \cite[Theorem 5]{Agaev}
and recorded in \cite[Theorem D]{Limonova}. The existence assertion
for general $2<p<r$ also follows by applying Talagrand's general
restriction theorem in $L^r$ and its accompanying decomposition;
see \cite[Theorem 19.3.8 and Lemma 19.3.11]{TalagrandBook}.

\begin{theorem}[$L^r$-normalized selection]
\label{thm:lr-selection}
Let $2<p<r<\infty$, let $A>0$, and let $N\ge2$. Suppose that
$\phi_1,\ldots,\phi_N$ are pairwise orthogonal functions on a probability
space $(X,\mu)$ satisfying $\|\phi_i\|_r\le1$ for every $i\in[N]$.
Set
\[
  \theta=\frac{2(r-p)}{p(r-2)},
  \qquad m=\left\lceil N^\theta\right\rceil.
\]
There is a constant $C_{p,r,A}$, depending only on $p,r,A$, such that
a uniformly chosen $m$-element subset $S\subset[N]$ satisfies
\begin{equation}\label{eq:lr-selection-conclusion}
  \left\|\sum_{i\in S}a_i\phi_i\right\|_p
  \le C_{p,r,A}\left(\sum_{i\in S}|a_i|^2\right)^{\frac12}
\end{equation}
simultaneously for all complex coefficient vectors, with probability
at least $1-N^{-A}$. In particular, such a subset exists with a
constant depending only on $p$ and $r$.
\end{theorem}

\begin{proof}
It suffices to consider $A\ge1$. Set
\[
  \alpha=1-\theta,\qquad \delta=\frac{m}{N},\qquad
  \ell=\log(eN),\qquad
  \lambda_1=1,\qquad \lambda_2=1+\frac{\alpha}{2}\log N.
\]
Then $0<\alpha<1$, $\delta\le2N^{-\alpha}$, and
$1\le\lambda_1,\lambda_2\le\ell$.
Orthogonality and $\|\phi_i\|_2\le\|\phi_i\|_r\le1$ give
\eqref{eq:bessel}. Apply \eqref{eq:selection-polynomial} with
$\lambda_1,\lambda_2$. The estimates for
$\beta_\delta(\lambda_1)$ and $\beta_\delta(\lambda_2)$ used in
Section~\ref{sec:main-proof} apply with this value of $\alpha$.
Testing against $g_t$ from \eqref{e:gt}, as in
\eqref{eq:duality-tail} and \eqref{eq:tail-algebra}, gives the
following bounds on one event of probability at least
$1-\frac{1}{2N^A}$. For every coefficient vector with
$\sum_{i\in S}|a_i|^2\le1$, put
\[
  f=\sum_{i\in S}a_i\phi_i,\qquad
  F(t)=\mu\{|f|>t\}\quad(t>0).
\]
Then
\begin{equation}\label{eq:lr-first-tail}
  F(t)\le C_{p,r,A}\min\left\{
       \ell^{\frac r2}t^{-r}+\delta t^{-2},\;
       t^{-r}+\frac{N^{-\frac{\alpha}{2}}}{\ell}t^{-2}
       \right\}.
\end{equation}
In addition, $\|f\|_r\le\sum_{i\in S}|a_i|\le\sqrt m$ gives
$F(t)\le m^{\frac r2}t^{-r}$.
Combining this with the first bound in \eqref{eq:lr-first-tail},
using $\min\{a+b,c\}\le a+\min\{b,c\}$ for nonnegative $a,b,c$,
gives
\[
  F(t)\le C_{p,r,A}\left(
      \ell^{\frac r2}t^{-r}
      +\min\{\delta t^{-2},m^{\frac r2}t^{-r}\}\right).
\]

Set $L=\ell^{\frac{r}{2(r-p)}}$. In the layer-cake integral, use
$F(t)\le1$ below $1$, the second bound in
\eqref{eq:lr-first-tail} on $(1,L)$, and the preceding bound above
$L$. Since $r>p$ and $\ell^{\frac r2}L^{p-r}=1$, this yields
\[
  \|f\|_p^p\le C_{p,r,A}\left(
     1+\frac{N^{-\frac{\alpha}{2}}}{\ell}L^{p-2}
     +\int_0^\infty t^{p-1}
       \min\{\delta t^{-2},m^{\frac r2}t^{-r}\}\,dt\right).
\]
The middle term is bounded uniformly in $N$ because $\alpha>0$
and $L$ is a fixed power of $\log(eN)$. The two terms in the
minimum agree at $T=\sqrt m\,N^{\frac{1}{r-2}}$. Splitting there gives
\begin{align*}
  \int_0^\infty t^{p-1}
       \min\{\delta t^{-2},m^{\frac r2}t^{-r}\}\,dt
  &=\frac{\delta T^{p-2}}{p-2}
       +\frac{m^{\frac r2}T^{p-r}}{r-p}\\
  &=\left(\frac{1}{p-2}+\frac{1}{r-p}\right)
       m^{\frac p2}N^{-\frac{r-p}{r-2}}\le C_{p,r},
\end{align*}
where the last step uses $m\le2N^\theta$ and
$\frac{\theta p}{2}=\frac{r-p}{r-2}$. Thus $\|f\|_p\le C_{p,r,A}$
for every normalized coefficient vector on the same event. Rescaling
proves the theorem, and taking $A=1$ gives its existence assertion.
\end{proof}

The optimal cardinality exponent is already present in Gaposhkin's
example, reproduced in \cite[Theorem 5]{Agaev} and restated in
\cite[Theorem D]{Limonova}.

\subsection{Signal recovery}\label{sec:recovery}

The selection estimate gives an explicit quantification of the
recovery framework of Iosevich and Mayeli \cite{IosevichMayeli}.
The deterministic implications use an uncertainty principle and the
strict null-space condition for recovery by minimization of the
$\ell^1$ norm; these are classical mechanisms
\cite{DonohoStark,IosevichMayeli}. We record the consequences without
repeating their standard proofs.

Fix integers $N\ge2$ and $d\ge1$, and put
$G=\mathbb{Z}_N^d$, where $\mathbb{Z}_N=\mathbb{Z}/N\mathbb{Z}$,
and $D=|G|=N^d$. A signal is a function $h:G\to\mathbb{C}$.
We use normalized averages and the unitary Fourier transform
\[
  \mathbb{E}_{x\in G}u(x)=\frac{1}{D}\sum_{x\in G}u(x),
  \qquad
  \widehat h(\xi)=D^{-\frac{1}{2}}
       \sum_{x\in G}e^{-\frac{2\pi\mathrm{i}x\cdot\xi}{N}}h(x),
\]
where $x\cdot\xi$ is the coordinate dot product modulo $N$.
Write $\operatorname{supp}(h)=\{x\in G:h(x)\ne0\}$ and
$\|h\|_{\ell^1(G)}=\sum_{x\in G}|h(x)|$.
A signal is $s$-sparse if its support has at most $s$ elements.
For a set $S\subset G$ of missing frequencies, the observed data
are $\widehat h$ on $S^c=G\setminus S$. Define
\[
  \mathcal{V}_S=
  \{h:G\to\mathbb{C}:\operatorname{supp}(\widehat h)\subset S\}.
\]
Thus two signals have the same observed data exactly when their
difference belongs to $\mathcal{V}_S$.

Applying Corollary~\ref{cor:polynomial-probability} to the $D$
orthonormal characters on $G$ with normalized counting measure
and using Parseval gives the following consequence.

\begin{lemma}\label{lem:finite-selected-space}
Fix $q>2$ and $A>0$. There is a constant $K=K(q,A)\ge1$,
independent of $N$ and $d$, such that a uniformly chosen subset
$S\subset G$ of cardinality $\lceil D^{\frac{2}{q}}\rceil$
satisfies
\begin{equation}\label{eq:finite-lambda}
  \left(\mathbb{E}_{x\in G}|h(x)|^q\right)^{\frac{1}{q}}
  \le K\left(\mathbb{E}_{x\in G}|h(x)|^2\right)^{\frac{1}{2}}
  \qquad(h\in\mathcal{V}_S)
\end{equation}
on an event of probability at least $1-D^{-A}=1-N^{-dA}$.
The inequality holds for all $h\in\mathcal{V}_S$ on the same event.
\end{lemma}

The standard uncertainty and null-space implications of
\eqref{eq:finite-lambda} also give a quantitative refinement of the recovery statement
\cite[Corollary 4.2]{IosevichMayeli}, see also \cite{DonohoStark}, replacing the qualitative success probability $1-o(1)$ with a super-polynomial bound.




\section*{Acknowledgments}

Alex Iosevich acknowledges partial support from the National Science
Foundation under grant DMS-2506858. Ben Krause acknowledges partial
support from the Engineering and Physical Sciences Research Council
through New Investigator Award EP/W010275/2 and from UK Research and
Innovation through Horizon Europe Guarantee grant EP/Y007336/1 for
his ERC Starting Grant project PCMEA (Pointwise Convergence of Multiple
Ergodic Averages).

\end{document}